\documentclass[12pt]{article}
\usepackage{amsmath,latexsym,amsthm}
\usepackage[dvips]{graphicx}
\usepackage{graphics}
\usepackage{amssymb}
\usepackage{color}
\usepackage{textcomp}
\usepackage{amsfonts}
\usepackage{amssymb}
\usepackage{longtable}
\usepackage{fancyhdr}
\usepackage[T2A]{fontenc}
\usepackage[utf8]{inputenc}

\newtheorem{lemma}{Lemma}
\newtheorem{theorem}{Theorem}

\makeatletter

\begin{document}

\begin{center}
{\Large\bf The Method of Potentials for the Airy Equation of Fractional Order}
\end{center}

\begin{center}
{\sc  Rakhimov~Kamoladdin~Urinbayevich}\\
{\it National University of Uzbekistan, Tashkent, Uzbekistan\\
}
e-mail: {\tt kamoliddin\_ru@inbox.ru }

\end{center}

\begin{abstract}
In this work, initial-boundary value problems for the time-fractional Airy equation are considered on different intervals. We studied the properties of potentials for this equation and, using these properties, found solutions to the considered problems. The uniqueness of the solution to the problem is proved using an analogue of the Gronwall–Bellman inequality and an a priori estimate.

\emph{\textbf{Keywords:} Time-fractional Airy equation, IBVP, fundamental solutions, integral representation.}

\emph{\textbf{Mathematics Subject Classification (2020):} 35R11, 35A08, 35A09, 35C15.}
\end{abstract}

\section*{Introduction}
In this paper we investigate initial-boundary value problems (IBVP) for time-fractional third order differential equation with multiple characteristics on the line. In literatures the analogue of this equation with whole order time derivative is also called the linearized Korteweg - de Vries equation \cite{Maqsad3,Calvante1,Maqsad2,Seifert1} or Airy equation \cite{Noja1}. Motivation of such study comes from wide range of applications of such IBVP on the modeling and simulation of the problems in physics \cite{Hilfer1}, mechanics \cite{Mainardi1}, biology and other fields of applied sciences \cite{Kilbas1,Metzler,Phys1,Phys2,Phys3}.

The theory of potentials for classical Airy equations considered in \cite{Jurayev,LC01}. They constructed the fundamental solutions of the equation and investigated some properties of the potentials with respect to these solutions. Using the properties of the potentials they found solutions of the some IVBP for the classical Airy equation. The fundamental solution of time-fractional Airy equation is constructed in \cite{AVP02}. In this paper they investigated some properties of potentials and give the solution of the Cauchy problem.

In \cite{Jurayev} constructed the exact solution of the problems for classical Airy equation on metric lines. They find the solution by using potential method. In \cite{AVP02} A.V. Pskhu investigated third order differential equation with fractional derivative and constructed a fundamental solution. He also solved the Cauchy problem for this equation in this work. The Cauchy problem for time fractional Airy equation on star-shaped graph with semi infinite bonds is considered in \cite{Bulletin}. In this work they constructed one fundamental solution of the equation in addition to one that constructed in \cite{AVP02}. Also, some additional properties of the potential are investigated and using this results the exact solution of the problem is constructed.

In this paper we consider three initial-boundary value problem on finite and semi-infinite intervals for the linearized Korteweg–de Vries equation with time-fractional derivative. To prove the existence and uniqueness of the solution we will use method of potentials and the method of energy integrals. Furthermore, we need to investigate some new properties of the potensials of the time-fractional Airy equation.

\section{Formulation of problems and some properties of potentials}

\subsection{Formulation of the problems}
The operator
\begin{equation}\label{hosila}
\partial_{0t}^{\alpha }f\left( t \right):=\frac{1}{\Gamma \left( 1-\alpha  \right)}\int\limits_{0}^{t}{\frac{f'\left( \tau  \right)}{{{\left( t-\tau  \right)}^{\alpha }}}d\tau };\,\,\,0<\alpha <1
\end{equation} is called Caputo fractional derivative. And inverse of this operator
\begin{equation}\label{integral}
J_{0,t}^{\alpha }f\left( t \right):=\frac{1}{\Gamma \left( \alpha  \right)}\int\limits_{0}^{t}{\frac{f\left( \tau  \right)}{{{\left( t-\tau  \right)}^{1-\alpha }}}d\tau };\,\,\,0<\alpha <1
\end{equation} is called fractional integral  \cite{AVP01}.

We consider the equation with a partial time derivative
\begin{equation}\label{tenglama}	
\partial_{0t}^{\alpha }u-\frac{{{\partial }^{3}}u}{\partial {{x}^{3}}}=f\left( x,t \right)
\end{equation} 	
on the different domains.

Put $E_1=\{(x,t): 0< x< 1, 0< t\leq T\}$,  $E_2=\{(x,t): 0< x<\infty, 0< t\leq T\}$,  $E_3=\{(x,t): -\infty<x< 0, 0< t\leq T\}$ with $T>0$.

\textbf{Problem 1.}
To find the solution $u(x,t)$ of the equation (\ref{tenglama}) from the following class of functions: $W_1=\left\{ u(x,t):u(x,t)\in C^{2,1}(E_1), u(x,t)\in C(\overline{E_1}), u_x\in C([0,1)\times (0,T])\right\}$, that satisfies the following initial condition
\begin{equation}\label{boshshart}
u\left( x,0 \right)=0,
\end{equation}
and boundary conditions
\begin{equation}\label{chegshart}
u\left( 0,t \right)={{\varphi }_{1}}\left( t \right);\,\,\frac{\partial u}{\partial x}( 0,t)={{\varphi }_{2}}\left( t \right)\,\,
\end{equation}
and
\begin{equation}\label{shart}
u\left( 1,t \right)={{\varphi }_{3}}\left( t \right).
\end{equation}

\textbf{Problem 2.}
To find the solution $u(x,t)$ of the equation (\ref{tenglama}) from the following class of functions: $W_2=\left\{ u(x,t):u(x,t)\in C^{2,1}(E_2), u_x (x,t)\in C(\overline{E_2}),\right\}$, that satisfies the following initial condition
\begin{equation}\label{boshshart1}
u\left( x,0 \right)=0,
\end{equation}
and following boundary conditions
\begin{equation}\label{chegshart1}
u(0,t)=\psi_{1}(t);\,\,\frac{\partial u}{\partial x}(0,t)=\psi_{2}(t).
\end{equation}

\textbf{Problem 3.}
Another problem is to find the solution $u(x,t)$ of the equation (\ref{tenglama}) from the following class of functions: $W_3=\left\{ u(x,t):u(x,t)\in C^{2,1}(E_3), u_{xx}(x,t)\in C(\overline{E_3}) \right\}$, that satisfies the following initial condition
\begin{equation}\label{boshshart2}
u\left( x,0 \right)=0,
\end{equation}
and following boundary condition
\begin{equation}\label{chegshart2}
\frac{\partial^2 u}{\partial x^2}(0,t)=\psi(t).
\end{equation}

\subsection{Potential theory}
The function
\begin{equation}\label{Wright}
\phi (\rho ,\mu ;z)=\sum\limits_{n=0}^{\infty }{\frac{{{z}^{n}}}{n!\Gamma (\rho n+\mu )}},\,\,\,\,-1<\rho <0\,,\,\mu \in \mathbb{C}
\end{equation}
is called Wright function \cite{Gorenflo}.
Wright function can be represented as
\begin{equation} \label{Wright2} \phi(\lambda,\mu;z)=\frac{1}{2\pi i}\int \limits_{Ha}^{}e^{\sigma+z\sigma^{-\lambda }}\frac{d\sigma}{\sigma^\mu},  \lambda>-1,   \mu \in \mathbb{C}\end{equation}
where the integral is taken along the Hankel contour $Ha$, i.e. a loop which starts and ends at $-\infty$ and encircles the circular disk $|\zeta|\leq|z|^{1/\omega}$ in the positive sense: $-\pi\leq arg\zeta \leq\pi$ on $Ha$ \cite{Mainardi1}.

The following auxiliary M-Wright functions are introduced \cite{FM01} \begin{equation} \label{MWright1}
F_\nu(z):=\phi(-\nu,0;-z),  0<\nu<1,
\end{equation}
\begin{equation} \label{MWright2}
M_\nu(z):=\phi(-\nu,1-\nu;-z),  0<\nu<1.
\end{equation}
They are related with following relation \begin{equation} \label{MWright3}
F_\nu(z)=\nu zM_\nu(z).
\end{equation}

 We have following estimate \cite{AVP02}
\begin{equation} \label{baho}
\left| \phi \left( -\delta ,\varepsilon ;z \right) \right|\le C\exp \left( -\nu {{\left| z \right|}^{\frac{1}{1-\delta }}} \right),\,\,\,C=C\left( \delta ,\varepsilon ,\nu  \right)
\end{equation}
 where $\nu <\left( 1-\delta  \right){{\delta }^{\frac{\delta }{1-\delta }}}\cos \frac{\pi -\left| \arg z \right|}{1-\delta },\,\,\frac{1+\delta }{2}\pi <\left| \arg z \right|\le \pi$.

For constructing the solution of the problem we need to get a special functions, which called fundamental solutions.  The fundamental solutions were founded in the following forms \cite{Bulletin, AVP02}
\begin{equation} \label{fundamental1}
G_{\alpha}^{2\alpha/3}(x,t)=\frac{1}{3t^{1-2\alpha/3}}
    \begin{cases}
  & \phi (-\frac{\alpha }{3},\frac{2\alpha }{3};\frac{x}{{{t}^{\alpha /3}}});\,\,\,\,\,\,\,\,x<0 \\
 & -2\operatorname{Re}[{{e}^{2\pi i/3}}\phi (-\frac{\alpha }{3},\frac{2\alpha }{3};{{e}^{2\pi i/3}}\frac{x}{{{t}^{\alpha /3}}})];\,\,\,\,x>0 \\
    \end{cases}
\end{equation}
and
\begin{equation} \label{fundamental2}
V_{\alpha }^{2\alpha /3}(x,t)=\frac{1}{3{{t}^{1-2\alpha /3}}}\operatorname{Im}[{{e}^{2\pi i/3}}\phi (-\frac{\alpha }{3},\frac{2\alpha }{3};{{e}^{2\pi i/3}}\frac{x}{{{t}^{\alpha /3}}})];\,\,\,\,x>0.
\end{equation}

Based on solutions (\ref{fundamental1}) and (\ref{fundamental2}), we construct the potentials and study their properties.

\begin{lemma}
Let $\tau_k \in C[0,T]$, $k=1,2$. Define
\[
w_1(x,t)=\int_0^t G_\alpha^{2\alpha/3}(x-a,t-\eta)\,\tau_1(\eta)\,d\eta,
\]
\[
w_2(x,t)=\int_0^t V_\alpha^{2\alpha/3}(x-a,t-\eta)\,\tau_2(\eta)\,d\eta.
\]
Then:

(i) The functions $w_k(x,t)$, $k=1,2$, satisfy the homogeneous equation
\[
\partial_{0t}^{\alpha} w_k(x,t)-\frac{\partial^3 w_k(x,t)}{\partial x^3}=0,
\qquad x\neq a,\ t>0.
\]

(ii) For all $x\in\mathbb R$,
\[
\lim_{t\to 0+} w_k(x,t)=0,\qquad k=1,2.
\]
\end{lemma}

\begin{proof}
We prove the statement for $w_1$; the proof for $w_2$ is completely analogous, since the kernel $V_\alpha^{2\alpha/3}(x,t)$ possesses the same differential properties as $G_\alpha^{2\alpha/3}(x,t)$.

\medskip
\noindent

It is known (see \cite{AVP02}) that the kernel $G_\alpha^{\mu}(x,t)$ satisfies the identities
\begin{equation}\label{eqg1}
\partial_{0t}^{\alpha} G_\alpha^{\mu}(x,t)=G_\alpha^{\mu-\alpha}(x,t),    
\end{equation}
and 
\begin{equation}\label{eqg2}
\frac{\partial^3}{\partial x^3} G_\alpha^{\mu}(x,t)=G_\alpha^{\mu-\alpha}(x,t),    
\end{equation}
for $x\neq 0$, $t>0$. In particular, for $\mu=2\alpha/3$ we obtain
\[
\partial_{0t}^{\alpha} G_\alpha^{2\alpha/3}(x,t)
-
\frac{\partial^3}{\partial x^3} G_\alpha^{2\alpha/3}(x,t)=0,
\qquad x\neq 0,\ t>0.
\]

Since $\tau_1$ is continuous and bounded on $[0,T]$, and the kernel $G_\alpha^{2\alpha/3}(x,t)$ together with its required derivatives is locally integrable on $(0,T]$, differentiation under the integral sign is justified. Hence,
\[
\partial_{0t}^{\alpha} w_1(x,t)
=
\int_0^t \partial_{0t}^{\alpha}
G_\alpha^{2\alpha/3}(x-a,t-\eta)\,\tau_1(\eta)\,d\eta,
\]
and
\[
\frac{\partial^3 w_1(x,t)}{\partial x^3}
=
\int_0^t \frac{\partial^3}{\partial x^3}
G_\alpha^{2\alpha/3}(x-a,t-\eta)\,\tau_1(\eta)\,d\eta.
\]
Subtracting these relations and using the above differential identity for the kernel yields
\[
\partial_{0t}^{\alpha} w_1(x,t)
-
\frac{\partial^3 w_1(x,t)}{\partial x^3}
=0,
\qquad x\neq a,\ t>0.
\]

\medskip
\noindent

Since $\tau_1$ is bounded, we have
\[
|w_1(x,t)|
\le
\|\tau_1\|_{C[0,T]}
\int_0^t
\bigl|G_\alpha^{2\alpha/3}(x-a,s)\bigr|\,ds,
\qquad s=t-\eta.
\]

If $x\neq a$, then by the exponential-type estimate for the Wright function (see \cite{AVP02}),
\[
|G_\alpha^{2\alpha/3}(x-a,s)|
\le
C\, s^{-(1-2\alpha/3)}
\exp\!\left(-\nu\,\Bigl|\frac{x-a}{s^{\alpha/3}}\Bigr|^{\frac{1}{1-\alpha/3}}\right),
\]
which implies
\[
\int_0^t |G_\alpha^{2\alpha/3}(x-a,s)|\,ds \to 0
\quad \text{as } t\to 0+.
\]

If $x=a$, then
\[
|G_\alpha^{2\alpha/3}(0,s)|\le C s^{-(1-2\alpha/3)},
\]
and therefore
\[
\int_0^t s^{-(1-2\alpha/3)}\,ds
=
C\, t^{2\alpha/3}\to 0,
\qquad t\to 0+.
\]

Hence,
\[
\lim_{t\to 0+} w_1(x,t)=0
\]
for all $x\in\mathbb R$.

The proof for $w_2$ is identical.

\end{proof}

\medskip

\begin{lemma}
Let
\[
w_3(x,t)=\int_0^t \frac{\partial^2}{\partial x^2}
G_\alpha^{2\alpha/3}(x-a,t-\eta)\,\tau_3(\eta)\,d\eta
\]
and
\[
w_4(x,t)=\int_0^t \frac{\partial^2}{\partial x^2}
V_\alpha^{2\alpha/3}(x-a,t-\eta)\,\tau_4(\eta)\,d\eta,
\]
where $\tau_3,\tau_4\in C[0,T]$ are bounded. Then, for every fixed $t\in(0,T]$,
\[
\lim_{x\to a-0} w_3(x,t)=\frac13\,\tau_3(t),
\]
\[
\lim_{x\to a+0} w_3(x,t)=-\frac23\,\tau_3(t),
\]
and
\[
\lim_{x\to a+0} w_4(x,t)=0.
\]
\end{lemma}

\begin{proof}
We use the differentiation formula for the Wright function \cite{AVP02}
\begin{equation}\nonumber
\partial^{\nu}_{0t}
\left(
t^{\mu-1}
\phi\!\left(-\frac{\alpha}{3},\mu;\lambda t^{-\alpha/3}\right)
\right)
=
t^{\mu-\nu-1}
\phi\!\left(-\frac{\alpha}{3},\mu-\nu;\lambda t^{-\alpha/3}\right).
\end{equation}
which implies
\[
\frac{d^2}{dz^2}\phi\!\left(-\frac{\alpha}{3},\frac{2\alpha}{3};z\right)
=
\phi\!\left(-\frac{\alpha}{3},0;z\right).
\]
Therefore, for $s>0$ and $x<a$, we have
\[
\frac{\partial^2}{\partial x^2}G_\alpha^{2\alpha/3}(x-a,s)
=
\frac{1}{3s}\,
\phi\!\left(-\frac{\alpha}{3},0;\frac{x-a}{s^{\alpha/3}}\right),
\]
while for $x>a$,
\[
\frac{\partial^2}{\partial x^2}G_\alpha^{2\alpha/3}(x-a,s)
=
-\frac{2}{3s}\,
\textrm{Re}\left[
\phi\!\left(-\frac{\alpha}{3},0;
e^{2\pi i/3}\frac{x-a}{s^{\alpha/3}}
\right)
\right],
\]
and
\[
\frac{\partial^2}{\partial x^2}V_\alpha^{2\alpha/3}(x-a,s)
=
\frac{1}{3s}\,
\textrm{Im}\left[
\phi\!\left(-\frac{\alpha}{3},0;
e^{2\pi i/3}\frac{x-a}{s^{\alpha/3}}
\right)
\right].
\]

We first consider the left-hand limit of $w_3$.

\medskip
\noindent

For $x<a$, setting $s=t-\eta$, we write
\[
w_3(x,t)
=
\int_0^t
\frac{1}{3s}\,
\phi\!\left(-\frac{\alpha}{3},0;\frac{x-a}{s^{\alpha/3}}\right)
\tau_3(t-s)\,ds.
\]
Since $x-a<0$, introduce the change of variables
\[
u=\frac{a-x}{s^{\alpha/3}},
\qquad
s=\left(\frac{a-x}{u}\right)^{3/\alpha},
\qquad
\frac{ds}{s}=-\frac{3}{\alpha}\frac{du}{u}.
\]
Then
\[
w_3(x,t)
=
\frac{1}{\alpha}
\int_{\frac{a-x}{t^{\alpha/3}}}^{\infty}
\phi\!\left(-\frac{\alpha}{3},0;-u\right)
\frac{\tau_3\!\left(t-\left(\frac{a-x}{u}\right)^{3/\alpha}\right)}{u}\,du.
\]
Now split the integral as
\[
w_3(x,t)=J_1(x,t)+J_2(x,t),
\]
where
\[
J_1(x,t)=
\frac{1}{\alpha}
\int_{\frac{a-x}{t^{\alpha/3}}}^{\infty}
\phi\!\left(-\frac{\alpha}{3},0;-u\right)
\frac{
\tau_3\!\left(t-\left(\frac{a-x}{u}\right)^{3/\alpha}\right)-\tau_3(t)
}{u}\,du,
\]
\[
J_2(x,t)=
\frac{\tau_3(t)}{\alpha}
\int_{\frac{a-x}{t^{\alpha/3}}}^{\infty}
\phi\!\left(-\frac{\alpha}{3},0;-u\right)\frac{du}{u}.
\]

Since $\tau_3$ is continuous at $t$, for each fixed $u>0$,
\[
\tau_3\!\left(t-\left(\frac{a-x}{u}\right)^{3/\alpha}\right)\to \tau_3(t)
\qquad \text{as } x\to a-0.
\]
Moreover, by the standard exponential estimate (\ref{baho}) for the Wright function,
\[
\left|\phi\!\left(-\frac{\alpha}{3},0;-u\right)\right|
\le C e^{-\nu u^{\frac{1}{1-\alpha/3}}},
\]
hence
\[
\left|\phi\!\left(-\frac{\alpha}{3},0;-u\right)\right|\frac{1}{u}
\in L^1(0,\infty).
\]
Therefore, by the dominated convergence theorem,
\[
J_1(x,t)\to 0
\qquad \text{as } x\to a-0.
\]

For $J_2$, as $x\to a-0$ the lower limit tends to $0$. The following identity holds (see \cite{FM01,Podlubny})
\[
\int_0^\infty
\phi\!\left(-\frac{\alpha}{3},0;-u\right)\frac{du}{u}
=
\frac{\alpha}{3}.
\]
This identity follows from definition of M-Wright functions and the relation (\ref{MWright3}).

Using this identity we obtain
\[
J_2(x,t)\to \frac{1}{3}\tau_3(t).
\]
Thus,
\[
\lim_{x\to a-0} w_3(x,t)=\frac13\,\tau_3(t).
\]

\medskip
\noindent

Now we find the limit $\boldsymbol{x\to a+0}$ for $\boldsymbol{w_3}$.
For $x>a$, we have
\[
w_3(x,t)
=
-\frac{2}{3}
\int_0^t
\frac{1}{s}\,
\textrm{Re}\left[
\phi\!\left(-\frac{\alpha}{3},0;
e^{2\pi i/3}\frac{x-a}{s^{\alpha/3}}
\right)
\right]
\tau_3(t-s)\,ds.
\]
Now set
\[
u=\frac{x-a}{s^{\alpha/3}},
\qquad
s=\left(\frac{x-a}{u}\right)^{3/\alpha},
\qquad
\frac{ds}{s}=-\frac{3}{\alpha}\frac{du}{u}.
\]
Then
\[
w_3(x,t)
=
-\frac{2}{\alpha}
\int_{\frac{x-a}{t^{\alpha/3}}}^{\infty}
\textrm{Re}\left[
\phi\!\left(-\frac{\alpha}{3},0;e^{2\pi i/3}u\right)
\right]
\frac{\tau_3\!\left(t-\left(\frac{x-a}{u}\right)^{3/\alpha}\right)}{u}\,du.
\]
As before, splitting into the difference term and the principal term and applying dominated convergence, we get
\[
\lim_{x\to a+0} w_3(x,t)
=
-\frac{2\tau_3(t)}{\alpha}
\int_0^\infty
\textrm{Re}\left[
\phi\!\left(-\frac{\alpha}{3},0;e^{2\pi i/3}u\right)
\right]\frac{du}{u}.
\]
Using the identity (\ref{MWright3}), from
\[
\int_0^\infty
\textrm{Re}\left[
\phi\!\left(-\frac{\alpha}{3},0;e^{2\pi i/3}u\right)
\right]\frac{du}{u}
=
\frac{\alpha}{3},
\]
we obtain
\[
\lim_{x\to a+0} w_3(x,t)=-\frac23\,\tau_3(t).
\]

\medskip
\noindent

The limit $\boldsymbol{x\to a+0}$ for $\boldsymbol{w_4}$.
For $x>a$,
\[
w_4(x,t)
=
\frac13
\int_0^t
\frac{1}{s}\,
\textrm{Im}\left[
\phi\!\left(-\frac{\alpha}{3},0;
e^{2\pi i/3}\frac{x-a}{s^{\alpha/3}}
\right)
\right]
\tau_4(t-s)\,ds.
\]
With the same substitution
\[
u=\frac{x-a}{s^{\alpha/3}},
\]
we get
\[
w_4(x,t)
=
\frac{1}{\alpha}
\int_{\frac{x-a}{t^{\alpha/3}}}^{\infty}
\textrm{Im}\left[
\phi\!\left(-\frac{\alpha}{3},0;e^{2\pi i/3}u\right)
\right]
\frac{\tau_4\!\left(t-\left(\frac{x-a}{u}\right)^{3/\alpha}\right)}{u}\,du.
\]
Proceeding in the same way and applying the dominated convergence theorem, by continuity of $\tau_4$ and dominated convergence,
\[
\lim_{x\to a+0} w_4(x,t)
=
\frac{\tau_4(t)}{\alpha}
\int_0^\infty
\textrm{Im}\left[
\phi\!\left(-\frac{\alpha}{3},0;e^{2\pi i/3}u\right)
\right]\frac{du}{u}.
\]
Finally, by the corresponding Wright-function identity,
\[
\int_0^\infty
\textrm{Im}\left[
\phi\!\left(-\frac{\alpha}{3},0;e^{2\pi i/3}u\right)
\right]\frac{du}{u}
=0,
\]
and therefore
\[
\lim_{x\to a+0} w_4(x,t)=0.
\]

The proof is complete.
\end{proof}

\medskip

\begin{lemma}
Let $\tau_5\in C[a,b]$ and define
\[
w_5(x,t)=\int_a^b G_\alpha^{2\alpha/3}(x-\xi,t)\tau_5(\xi)\,d\xi.
\]
Then $w_5$ satisfies
\[
\partial_{0t}^{\alpha}w_5(x,t)-\frac{\partial^3 w_5(x,t)}{\partial x^3}=0,
\]
and
\[
\lim_{t\to0+} I_{0t}^{1-\alpha}w_5(x,t)=\tau_5(x), \qquad x\in(a,b).
\]
\end{lemma}

\begin{proof}
    
We use differential identities (\ref{eqg1}, \ref{eqg2}) for the kernel $G_\alpha^\mu$.

Hence, differentiating under the integral sign, we obtain
\[
\partial_{0t}^{\alpha}w_5(x,t)
=
\int_a^b G_\alpha^{-\alpha/3}(x-\xi,t)\tau_5(\xi)\,d\xi
=
\frac{\partial^3 w_5(x,t)}{\partial x^3}.
\]

Next, using the fractional integration identity
\[
I_{0t}^{1-\alpha}G_\alpha^{2\alpha/3}(x,t)=G_\alpha^{1-\alpha/3}(x,t),
\]
we have
\[
I_{0t}^{1-\alpha}w_5(x,t)
=
\int_a^b G_\alpha^{1-\alpha/3}(x-\xi,t)\tau_5(\xi)\,d\xi.
\]
By the scaling property of the kernel,
\[
G_\alpha^{1-\alpha/3}(x,t)
=
t^{-\alpha/3}g_\alpha\!\left(\frac{x}{t^{\alpha/3}}\right),
\]
where
\[
g_\alpha(y):=G_\alpha^{1-\alpha/3}(y,1), \qquad g_\alpha\in L^1(\mathbb R),
\qquad \int_{-\infty}^{\infty}g_\alpha(y)\,dy=1.
\]
Therefore,
\[
I_{0t}^{1-\alpha}w_5(x,t)
=
\int_{\frac{x-b}{t^{\alpha/3}}}^{\frac{x-a}{t^{\alpha/3}}}
g_\alpha(y)\,\tau_5(x-t^{\alpha/3}y)\,dy.
\]
Since $\tau_5$ is continuous on $[a,b]$, for every fixed $y$ we have
\[
\tau_5(x-t^{\alpha/3}y)\to \tau_5(x)
\qquad \text{as } t\to0+.
\]
Moreover,
\[
|g_\alpha(y)\tau_5(x-t^{\alpha/3}y)|
\le \|\tau_5\|_{C[a,b]} |g_\alpha(y)|,
\]
and $g_\alpha\in L^1(\mathbb R)$. Hence, by the dominated convergence theorem,
\[
\lim_{t\to0+}I_{0t}^{1-\alpha}w_5(x,t)
=
\tau_5(x)\int_{-\infty}^{\infty}g_\alpha(y)\,dy
=
\tau_5(x).
\]
The proof is complete.
\end{proof}

\medskip

\begin{lemma}
Let $f\in C([a,b]\times[0,T])$. Then the function
\[
w_6(x,t)=\int_0^t\!\!\int_a^b
G_\alpha^{2\alpha/3}(x-\xi,t-\eta)f(\xi,\eta)
\,d\xi d\eta
\]
is a solution of
\[
\partial_{0t}^{\alpha}u-\frac{\partial^3 u}{\partial x^3}=f(x,t),
\]
and satisfies
\[
I_{0t}^{1-\alpha}w_6(x,t)\to0\qquad (t\to0+).
\]
\end{lemma}

\begin{proof}
    
Using the differential identity for the fundamental kernel
\[
\partial_{0t}^{\alpha}G_\alpha^{2\alpha/3}(x,t)
=
\frac{\partial^3}{\partial x^3}G_\alpha^{2\alpha/3}(x,t),
\]
we differentiate under the integral sign and obtain
\[
\partial_{0t}^{\alpha}w_6(x,t)
=
\int_0^t\!\!\int_a^b
\partial_x^3 G_\alpha^{2\alpha/3}(x-\xi,t-\eta)f(\xi,\eta)
\,d\xi d\eta.
\]
Hence,
\[
\partial_{0t}^{\alpha}w_6(x,t)
=
\frac{\partial^3}{\partial x^3}w_6(x,t)
+
\lim_{\eta\to t}\int_a^b
G_\alpha^{1-\alpha/3}(x-\xi,t-\eta)f(\xi,\eta)d\xi.
\]

Using the standard representation of solutions via the fundamental kernel, we have
\[
G_\alpha^{1-\alpha/3}(x,t)
=
t^{-\alpha/3}\,
g_\alpha\!\left(\frac{x}{t^{\alpha/3}}\right),
\]
where
\[
g_\alpha(y):=G_\alpha^{1-\alpha/3}(y,1),
\qquad
g_\alpha\in L^1(\mathbb R).
\]

In particular, from relation \cite{AVP02}
\begin{equation}\label{eqg3} 
\int_{-\infty}^{+\infty} G_\alpha^{\mu}(x,t)\,dx
=
\frac{t^{\mu+\frac{\alpha}{3}-1}}{\Gamma\!\left(\mu+\frac{\alpha}{3}\right)}.
\end{equation}
with $\mu=1-\alpha/3$ it follows that
\[
\int_{\mathbb R} g_\alpha(y)\,dy = 1,
\]
and hence
\[
\int_{\mathbb R} G_\alpha^{1-\alpha/3}(x,t)\,dx = 1.
\]

Consequently, for any $f\in C([a,b]\times[0,T])$,
\[
\int_a^b G_\alpha^{1-\alpha/3}(x-\xi,t-\eta)f(\xi,\eta)\,d\xi
\;\longrightarrow\;
f(x,t)
\quad \text{as } \eta\to t.
\]
The convergence follows from the dominated convergence theorem and the continuity of $f$.

\end{proof}

\begin{lemma}
Let $f\in C([a,b]\times[0,T])$. Then the function
\[
w_6(x,t)=\int_0^t\int_a^b
G_\alpha^{2\alpha/3}(x-\xi,t-\eta)f(\xi,\eta)\,d\xi\,d\eta
\]
satisfies
\[
\partial_{0t}^{\alpha}w_6(x,t)-\frac{\partial^3 w_6(x,t)}{\partial x^3}=f(x,t),
\]
and
\[
w_6(x,0)=0.
\]
\end{lemma}

\begin{proof}
We first show that $w_6$ satisfies the equation
\[
\partial_{0t}^{\alpha}u-\frac{\partial^3u}{\partial x^3}=f(x,t).
\]

By the differential identities for the fundamental kernel,
\[
\partial_{0t}^{\alpha}G_\alpha^{2\alpha/3}(x,t)
=
G_\alpha^{-\alpha/3}(x,t),
\qquad
\frac{\partial^3}{\partial x^3}G_\alpha^{2\alpha/3}(x,t)
=
G_\alpha^{-\alpha/3}(x,t),
\]
we may differentiate under the integral sign and obtain
\[
\partial_{0t}^{\alpha}w_6(x,t)
=
\int_0^t\int_a^b
\partial_{0t}^{\alpha}G_\alpha^{2\alpha/3}(x-\xi,t-\eta)\,f(\xi,\eta)\,d\xi\,d\eta,
\]
that is,
\[
\partial_{0t}^{\alpha}w_6(x,t)
=
\int_0^t\int_a^b
G_\alpha^{-\alpha/3}(x-\xi,t-\eta)\,f(\xi,\eta)\,d\xi\,d\eta.
\]
Similarly,
\[
\frac{\partial^3w_6(x,t)}{\partial x^3}
=
\int_0^t\int_a^b
\frac{\partial^3}{\partial x^3}G_\alpha^{2\alpha/3}(x-\xi,t-\eta)\,f(\xi,\eta)\,d\xi\,d\eta
=
\int_0^t\int_a^b
G_\alpha^{-\alpha/3}(x-\xi,t-\eta)\,f(\xi,\eta)\,d\xi\,d\eta.
\]

Thus, the singular contribution at $\eta=t$ has to be taken into account. To extract this term, we apply the fractional integral operator $I_{0t}^{1-\alpha}$ to $w_6$:
\[
I_{0t}^{1-\alpha}w_6(x,t)
=
\int_0^t\int_a^b
G_\alpha^{1-\alpha/3}(x-\xi,t-\eta)\,f(\xi,\eta)\,d\xi\,d\eta.
\]

Now we use the scaling property of the kernel:
\[
G_\alpha^{1-\alpha/3}(x,t)
=
t^{-\alpha/3}\,
g_\alpha\!\left(\frac{x}{t^{\alpha/3}}\right),
\qquad
g_\alpha(y):=G_\alpha^{1-\alpha/3}(y,1).
\]
Moreover, by \eqref{eqg3} with $\mu=1-\alpha/3$,
\[
\int_{\mathbb R}g_\alpha(y)\,dy=1,
\]
so the family $\{G_\alpha^{1-\alpha/3}(\cdot,t)\}_{t>0}$ forms an approximate identity.

Therefore, for every $(x,t)\in(a,b)\times(0,T]$,
\[
\int_a^b
G_\alpha^{1-\alpha/3}(x-\xi,t-\eta)\,f(\xi,\eta)\,d\xi
\longrightarrow
f(x,t)
\qquad \text{as }\eta\to t.
\]
Hence,
\[
\partial_{0t}^{\alpha}w_6(x,t)-\frac{\partial^3w_6(x,t)}{\partial x^3}=f(x,t).
\]
Finally, the initial condition $w_6(x,0)=0$ follows directly from the definition of $w_6$.
\end{proof}

\section{Well-posedness of problems}
\subsection{Uniqueness of solutions}
\begin{theorem}\label{yagona}
The Problem 1 has at most one regular solution.
\end{theorem}
\textbf{Proof.} To prove this theorem we need following inequality from \cite{Ali1}
\begin{equation}\label{alitengsizlik}
\partial_{0t}^{\alpha }\int\limits_{a}^{b}{{{u}^{2}}(x,t)dx}\le 2\int\limits_{a}^{b}{u(x,t)\partial_{0t}^{\alpha }u(x,t)dx}.
\end{equation}
Substituting in this inequality $\partial_{0t}^{\alpha }u(x,t)$ for $\frac{{{\partial }^{3}}}{\partial {{x}^{3}}}u(x,t)+f\left( x,t \right)$ from  (\ref{tenglama}) and using Cauchy inequality we have
$$\partial_{0t}^{\alpha }\int\limits_{a}^{b}{{{u}^{2}}dx}\le 2\int\limits_{a}^{b}{u\partial_{0t}^{\alpha }udx}=2\int\limits_{a}^{b}{u(u_{xxx}+f)dx}=$$
$$=2\int\limits_{a}^{b}{ud{{u}_{xx}}}+2\int\limits_{a}^{b}{ufdx}=2\left( \left. u{{u}_{xx}} \right|_{a}^{b}-\int\limits_{a}^{b}{ud{{u}_{xx}}} \right)+2\int\limits_{a}^{b}{ufdx}\le $$
$$\le \left. \left( 2u{{u}_{xx}}-{{\left( {{u}_{x}} \right)}^{2}} \right) \right|_{a}^{b}+4\int\limits_{a}^{b}{{{u}^{2}}dx}+\frac{1}{4}\int\limits_{a}^{b}{{{f}^{2}}dx}.$$
So we took following estimate
\begin{equation}\label{umumtsiz}
\partial_{0t}^{\alpha }\int\limits_{a}^{b}{{{u}^{2}}dx}\le \left. \left( 2u{{u}_{xx}}-{{\left( {{u}_{x}} \right)}^{2}} \right) \right|_{a}^{b}+4\int\limits_{a}^{b}{{{u}^{2}}dx}+\frac{1}{4}\int\limits_{a}^{b}{{{f}^{2}}dx}.
\end{equation}
Substituting in this inequality $a$ and $b$ for bounds of $E$ we obtain
\begin{equation}\label{teo1}
\partial_{0t}^{\alpha }{{\left\| u \right\|}^{2}}\le 4{{\left\| u \right\|}^{2}}+\frac{1}{4}{{\left\| f \right\|}^{2}},
\end{equation}
where ${{\left\| u \right\|}^{2}}=\int\limits_{0}^{1}{{{u}^{2}}\left( x,t \right)dx}$ and ${{\left\| f \right\|}^{2}}=\int\limits_{0}^{1}{{{f}^{2}}\left( x,t \right)dx}.$
From analogue of Grönwall–Bellman inequality \cite{Ali1}, we get the following relation:
\begin{equation}\label{te02k}
{{\left\| u \right\|}^{2}}\le \left\| u \right\|^{2}\cdot {{E}_{\alpha }}\left( 4{{t}^{\alpha }} \right)+\frac{\Gamma \left( \alpha  \right)}{4}{{E}_{\alpha ,\alpha }}\left( 4{{t}^{\alpha }} \right)J_{0t}^{\alpha }\left\| f \right\|,
\end{equation}
where ${{E}_{\alpha }}\left( z \right)=\sum\limits_{n=0}^{\infty }{\frac{{{z}^{n}}}{\Gamma \left( \alpha n+1 \right)}}$  and ${{E}_{\alpha ,\mu }}\left( z \right)=\sum\limits_{n=0}^{\infty }{\frac{{{z}^{n}}}{\Gamma \left( \alpha n+\mu  \right)}}$ are the Mittag-Leffler functions.
The uniqueness of solution follows from (\ref{te02k}).

\begin{theorem}\label{yagonalik2}
The Problem 2 has at most one regular solution.
\end{theorem}
\textbf{Proof.}
Substituting for $a$ and $b$ for $0$ and $+\infty$ respectively in the inequality (\ref{umumtsiz}), we have
\begin{equation}
\partial_{0t}^{\alpha }\int\limits_{0}^{+\infty}{{{u}^{2}}dx}\le \left. \left( 2u{{u}_{xx}}-{{\left( {{u}_{x}} \right)}^{2}} \right) \right|_{a}^{b}+4\int\limits_{0}^{+\infty}{{{u}^{2}}dx}+\frac{1}{4}\int\limits_{0}^{+\infty}{{{f}^{2}}dx}.
\end{equation}
Hence we get
\begin{equation}
\partial_{0t}^{\alpha }{{\left\| u \right\|}^{2}}\le 4{{\left\| u \right\|}^{2}}+\frac{1}{4}{{\left\| f \right\|}^{2}}.
\end{equation}
where ${{\left\| u \right\|}^{2}}=\int\limits_{0}^{+\infty}{{{u}^{2}}\left( x,t \right)dx}$ and ${{\left\| f \right\|}^{2}}=\int\limits_{0}^{+\infty}{{{f}^{2}}\left( x,t \right)dx}.$
From analogue of Grönwall–Bellman inequality, we get the following relation:
\begin{equation}\label{te02}
{{\left\| u \right\|}^{2}}\le \left\| u \right\|^{2}\cdot {{E}_{\alpha }}\left( 4{{t}^{\alpha }} \right)+\frac{\Gamma \left( \alpha  \right)}{4}{{E}_{\alpha ,\alpha }}\left( 4{{t}^{\alpha }} \right)J_{0t}^{\alpha }\left\| f \right\|.
\end{equation}
The uniqueness of solution follows from (\ref{te02}).
\begin{theorem}
The Problem 3 has at most one regular solution.
\end{theorem}
The proof of this theorem is similar to the previous ones.

\subsection{Existence of solutions}
\begin{theorem}\label{mavjud}
If ${{\varphi }_{j}}\left( t \right)\in C^1\left[ 0,T \right]$, ($j=1,2,3$) and $f\left( x,t \right)\in {{C}^{2}}(\overline{E_1})$, then the Problem 1 has a unique regular solution.
\end{theorem}
\textbf{Proof.} We look for solution in the form
$$u(x,t)=\int\limits_{0}^{t}{G_{\alpha }^{2\alpha /3}(x,t-\tau )\alpha (\tau )d\tau }+\int\limits_{0}^{t}{V_{\alpha }^{2\alpha /3}(x,t-\tau )\beta (\tau )d\tau }+$$
\begin{equation}\label{yechim}
+\int\limits_{0}^{t}{G_{\alpha }^{2\alpha /3}(x-1,t-\tau )\gamma (\tau )d\tau }+F(x,t),
\end{equation}
where
$$F(x,t)=\int\limits_{0}^{t}{\int_{0}^{1}{G_{\alpha }^{2\alpha /3}(x-\xi ,t-\tau )f(\xi ,\tau )d\xi }d\tau }.$$
Using  Lemma 1 it is easy to show that solution (\ref{yechim}) satisfies equation (\ref{tenglama}) and initial condition (\ref{boshshart}). Satisfying boundary conditions (\ref{chegshart}) we get

$$\int_{0}^{t}{\frac{\alpha \left( \tau  \right)+\frac{\sqrt{3}}{2}\beta \left( \tau  \right)}{3\Gamma \left( 2\alpha /3 \right){{\left( t-\tau  \right)}^{1-2\alpha /3}}}d\tau }={{\varphi }_{1}}\left( t \right)-\int_{0}^{t}{G_{\alpha }^{2\alpha /3}\left( -1,t-\tau  \right)\gamma \left( \tau  \right)d\tau }-F\left( 0,t \right)$$
and
$$\int_{0}^{t}{\frac{\alpha \left( \tau  \right)-\frac{\sqrt{3}}{2}\beta \left( \tau  \right)}{3\Gamma \left( \alpha /3 \right){{\left( t-\tau  \right)}^{1-\alpha /3}}}d\tau }+\frac{\partial }{\partial x}\int_{0}^{t}{G_{\alpha }^{\alpha /3}\left( -1,t-\tau  \right)\gamma \left( \tau  \right)d\tau }={{\varphi }_{2}}\left( t \right)-\frac{\partial }{\partial x}F\left( 0,t \right).$$
Taking into account definitions (\ref{hosila}), (\ref{integral})  and above equalities we obtain
\begin{equation}\label{syst1}
\frac{1}{3}\alpha \left( t \right)+\frac{\sqrt{3}}{6}\beta \left( t \right)+\int_{0}^{t}{{{K}_{1}}\left( t-\tau  \right)\gamma \left( \tau  \right)d\tau }=\partial_{0t}^{2\alpha /3}\left( {{\varphi }_{1}}\left( t \right)-F\left( 0,t \right) \right)
\end{equation}and
\begin{equation}\label{syst2}
\frac{1}{3}\alpha \left( t \right)-\frac{\sqrt{3}}{6}\beta \left( t \right)+\int_{0}^{t}{{{K}_{2}}\left( t-\tau  \right)\gamma \left( \tau  \right)d\tau }=\partial_{0t}^{\alpha /3}\left( {{\varphi }_{2}}\left( t \right)-\frac{\partial }{\partial x}F\left( 0,t \right) \right)
\end{equation}
where ${{K}_{1}}\left( t-\tau  \right)=\partial_{0t}^{2\alpha /3}G_{\alpha }^{2\alpha /3}\left( -1,t-\tau  \right)$ and ${{K}_{2}}\left( t-\tau  \right)=\partial_{0t}^{\alpha /3}\frac{\partial }{\partial x}G_{\alpha }^{\alpha /3}\left( -1,t-\tau  \right).$
Analogously, satisfying condition (\ref{shart}) and using the properties of fractional derivatives we get
\begin{equation}\label{syst3}
\int\limits_{0}^{t}{{{K}_{3}}\left( t-\tau  \right)\alpha \left( \tau  \right)d\tau }+\int\limits_{0}^{t}{{{K}_{4}}\left( t-\tau  \right)\beta \left( \tau  \right)d\tau }+\frac{1}{3}\gamma \left( t \right)=\partial_{0t}^{2\alpha /3}\left( {{\varphi }_{3}}\left( t \right)-F\left( 1,t \right) \right)
\end{equation}
where ${{K}_{3}}\left( t-\tau  \right)=\partial_{0t}^{2\alpha /3}G_{\alpha }^{2\alpha /3}\left( 1,t-\tau  \right)$  and ${{K}_{4}}\left( t-\tau  \right)=\partial_{0t}^{2\alpha /3}V_{\alpha }^{2\alpha /3}\left( 1,t-\tau  \right)$.

Above we use the asymptotic estimate (\ref{baho}), and get
\begin{equation}\label{kbaho}|K_j(t-\tau)|<C,\,\,\,j=1,2,3,4.
\end{equation}
So we obtained following system of integral equations (\ref{syst1}) - (\ref{syst3}) with respect to unknowns $\Phi(t)={{\left( \alpha \left( t \right),\beta \left( t \right),\gamma \left( t \right) \right)}^{T}}$
\begin{equation}\label{syst}	
\Phi (t)=\int\limits_{0}^{t}{K(t-\tau)\Phi (\tau)d\tau}+F(t)
\end{equation}
where $$K(t-\tau)=-{{\left( \begin{matrix}
   \frac{1}{3} & \frac{\sqrt{3}}{6} & 0  \\
   \frac{1}{3} & -\frac{\sqrt{3}}{6} & 0  \\
   0 & 0 & 1  \\
\end{matrix} \right)}^{-1}}\left( \begin{matrix}
   0 & 0 & {{K}_{1}}\left( t-\tau  \right)  \\
   0 & 0 & {{K}_{2}}\left( t-\tau  \right)  \\
   {{K}_{3}}\left( t-\tau  \right) & {{K}_{4}}\left( t-\tau  \right) & 0  \\
\end{matrix} \right)$$ and
$$F\left( t \right)={{\left( \begin{matrix}
   \frac{1}{3} & \frac{\sqrt{3}}{6} & 0  \\
   \frac{1}{3} & -\frac{\sqrt{3}}{6} & 0  \\
   0 & 0 & 1  \\
\end{matrix} \right)}^{-1}}\left( \begin{matrix}
   \partial_{0t}^{2\alpha /3}\left( {{\varphi }_{1}}\left( t \right)-F\left( 0,t \right) \right)  \\
   \partial_{0t}^{\alpha /3}\left( {{\varphi }_{2}}\left( t \right)-F\left( 0,t \right) \right)  \\
   \partial_{0t}^{2\alpha /3}\left( {{\varphi }_{3}}\left( t \right)-F\left( 1,t \right) \right)  \\
\end{matrix} \right).$$
From (\ref{kbaho}) follows that integral equations system (\ref{syst}) has unique solutions. This solution has a form
\begin{equation}\label{yechim1}
\Phi \left( t \right)=\int\limits_{0}^{t}{R\left( t-\tau  \right)K\left( t-\tau  \right)d\tau }+\tilde{F}\left( t \right),
\end{equation}
where $R(t-\tau)$  is the resolvent of system (\ref{syst}).

\begin{theorem}\label{mavjud2}

If ${{\psi }_{j}}\left( t \right)\in C^1\left[ 0,T \right]$, ($j=1,2$) and $f\left( x,t \right)\in {{C}^{2}}(\overline{E_2})$, then the Problem 2 has a unique regular solution on the form
\begin{equation}\label{yechim2}
u(x,t)=\int\limits_{0}^{t}{G_{\alpha }^{2\alpha /3}(x,t-\tau )\lambda (\tau )d\tau }+\int\limits_{0}^{t}{V_{\alpha }^{2\alpha /3}(x,t-\tau )\mu (\tau )d\tau }+Q(x,t),
\end{equation}
where
$$Q(x,t)=\int\limits_{0}^{t}{\int_{0}^{+\infty}{G_{\alpha }^{2\alpha /3}(x-\xi ,t-\tau )f(\xi ,\tau )d\xi }d\tau },$$
\begin{equation}\label{lambda2}
\lambda(t)=\frac{3}{2}\partial_{0t}^{2\alpha /3}\left( {{\psi }_{1}}\left( t \right)-Q\left( 0,t \right) \right)+\frac{3}{2}\partial_{0t}^{\alpha /3}\left( {{\psi }_{2}}\left( t \right)-\frac{\partial }{\partial x}Q\left( 0,t \right) \right)
\end{equation} and
\begin{equation}\label{myu2}
\mu(t)=\sqrt{3}\partial_{0t}^{2\alpha /3}\left( {{\psi }_{1}}\left( t \right)-Q\left( 0,t \right) \right)-\sqrt{3}\partial_{0t}^{\alpha /3}\left( {{\psi }_{2}}\left( t \right)-\frac{\partial }{\partial x}Q\left( 0,t \right) \right).
\end{equation}
\end{theorem}
\textbf{Proof.}
We look for solution in the form (\ref{yechim2}).
Using  Lemma (1) it is easy to show that solution (\ref{yechim2}) satisfies equation (\ref{tenglama}) and initial condition (\ref{boshshart1}). Satisfying boundary conditions (\ref{chegshart1}) we get

$$\int_{0}^{t}{\frac{\lambda(\tau)+\frac{\sqrt{3}}{2}\mu(\tau)}{3\Gamma \left( 2\alpha /3 \right){{\left( t-\tau  \right)}^{1-2\alpha /3}}}d\tau }={{\psi }_{1}}(t )-Q(0,t),$$
and
$$\int_{0}^{t}{\frac{\lambda(\tau)-\frac{\sqrt{3}}{2}\mu(\tau)}{3\Gamma \left( 2\alpha /3 \right){{\left( t-\tau  \right)}^{1-2\alpha /3}}}d\tau }={{\psi }_{2}}(t )-\frac{\partial}{\partial x}Q(0,t),$$

Taking into account definitions of fractional derivative and integral we obtain
\begin{equation}\label{sistema11}
\lambda(t)-\frac{\sqrt{3}}{2}\mu(t)=3\partial_{0t}^{2\alpha/3}\left({{\psi }_{1}}(t )-Q(0,t)\right)
\end{equation}
and
\begin{equation}\label{sistema12}
\lambda(t)+\frac{\sqrt{3}}{2}\mu(t)=3\partial_{0t}^{2\alpha/3}\left({{\psi }_{2}}(t )-\frac{\partial}{\partial x}Q(0,t)\right).
\end{equation}

So we obtained following system of equations with respect to unknowns

\begin{equation}\nonumber \label{sistemaa}
\begin{cases}
 \lambda(t)+\frac{\sqrt{3}}{2}\mu(t)=3\partial_{0t}^{2\alpha /3}\left( {{\psi }_{1}}\left( t \right)-Q\left( 0,t \right) \right) \\
 \lambda(t)-\frac{\sqrt{3}}{2}\mu(t)=3\partial_{0t}^{\alpha /3}\left( {{\psi }_{2}}\left( t \right)-\frac{\partial }{\partial x}Q\left( 0,t \right) \right). \\
\end{cases}
\end{equation}

Solving this system we get (\ref{lambda2}) and (\ref{myu2}) form of unknown functions.
Substituting unknown functions $\lambda(t)$ and $\mu(t)$ in (\ref{yechim2}) for above functions we have hole form of solution.

\begin{theorem}\label{mavjud3}
If $\psi\in C^1\left[ 0,T \right]$ and $f\left( x,t \right)\in {{C}^{2}}(\overline{E_3})$, then the Problem 3 has a unique regular solution on the form
\begin{equation}\label{yechim03}
u(x,t)=\frac{3}{2}\int\limits_{0}^{t}{G_{\alpha }^{2\alpha /3}(x,t-\tau )\left( \frac{{{\partial }^{2}}}{\partial {{x}^{2}}}R(x,\tau )-\psi \left( \tau  \right) \right)d\tau }+R(x,t)
\end{equation}
where
$$R(x,t)=\int\limits_{0}^{t}{\int_{-\infty}^{0}{G_{\alpha }^{2\alpha /3}(x-\xi ,t-\tau )f(\xi ,\tau )d\xi }d\tau }.$$
\end{theorem}

\textbf{Proof.} We look for solution in the form

\begin{equation}\label{yechim3}
u(x,t)=\int\limits_{0}^{t}{G_{\alpha }^{2\alpha /3}(x,t-\tau )\theta (\tau )d\tau }+R(x,t).
\end{equation}

Using  Lemma (1) it is easy to show that solution (\ref{yechim3}) satisfies equation (\ref{tenglama}) and initial condition (\ref{boshshart2}). Let us satisfy boundary condition (\ref{chegshart2})
$$
\underset{x\to -0}{\mathop{\lim }}\,\frac{\partial^2}{\partial x^2}u(x,t)=\underset{x\to -0}{\mathop{\lim }}\,\int\limits_{0}^{t}\frac{\partial^2}{\partial x^2}{G_{\alpha }^{2\alpha /3}(x,t-\tau )\theta (\tau )d\tau }+\frac{\partial^2}{\partial x^2}R(0,t)=\psi(t).
$$
Using Lemma (2) we get
$$\theta(t)=3\left(\psi(t)-\frac{\partial^2}{\partial x^2}R(0,t) \right).
$$
Substituting $\theta(\tau)$ in  \eqref{yechim3}  we get solution in the form (\ref{yechim03}). From the conditions of the theorem and the properties of potentials given above, it follows that the function given by (\ref{yechim03}) is a regular solution of the Problem 3.

In this work, we studied well-posedness of the initial boundary problem for time-fractional Airy equation on the different bounded and semi-infinite intervals. A-priory estimates for the solutions of considered problems are obtained using analogue of Grönwall–Bellman inequality and uniqueness of solutions are proved. We investigated the properties of potentials and solutions of problems are founded in the form of potentials sum.

\end{document}